\documentclass {amsart}

\newtheorem{theorem}{Theorem}[section]
\newtheorem{lemma}[theorem]{Lemma}
\newtheorem{proposition}[theorem]{Proposition}

\usepackage{amssymb}
\usepackage{amsmath}

\usepackage[all]{xy}
\usepackage[bookmarks]{hyperref}

\newcommand{\pp}{\mathbb P}
\newcommand{\calO}{\mathcal{O}}
\newcommand{\calE}{\mathcal{E}}

\newcommand{\st}{\overline}

\begin{document}

\title[bicanonical map]{Surfaces with 
 $p_g=q=1$, $K^2=8$ and nonbirtional
 bicanonical map}
\author{
Giuseppe Borrelli
}
\address{
\
\newline 
Universidade federal de Pernambuco,\newline
Departamento de Matem\' atica, \newline
Cidade Universitaria, %\newline
50670-901
Recife (PE), Brasil
\
\newline
}
\curraddr{ } 
\email{borrelli@dmat.ufpe.br}

\begin{abstract}
We prove that if the bicanonical map of 
a minimal surface
 of general type $S$ with $p_{g}=q=1$
and $K_{S}^2=8$ is non birational,
then
it is a double cover onto a rational surface.

An application of this theorem is the 
 complete classification of minimal
surfaces of general type 
with $p_{g}=q=1$,  $K_{S}^2=8$ 
and nonbirational bicanonical map.
\end{abstract}

\maketitle

%\section{Introduction}

Let $S$ be a smooth surface 
and $K_{S}$ a canonical divisor of $S$.
The $m-th$ canonical map is the rational
 map $\varphi_{mK_{S}}$ associated to
the complete linear system $\vert mK_{S}\vert$.
Investigating the behavior of $\varphi_{mK_{S}}$ ($m>0$) 
is a useful tool when 
studying  surfaces of general type, 
i.e. surfaces for which   $\varphi_{mK_{S}}$ is generically 
finite for some $m\geq 1$.

The behavior of the maps $\varphi_{mK_{S}}$
is well understood  when $m\geq 3$
(cf. \cite{C2} for a survey). 
For instance, we know that
 if $m\geq5$, then $\varphi_{2K_{S}}$ maps $S$ birationally
 onto a normal surface (cf. \cite{bom}).

Let $S$ be a (smooth) minimal surface of general type. 
Assume that the bicanonical (i.e. $m=2$) map is non birational.
Then $S$ may or may not admit a fibration whose general 
fiber is a smooth curve of genus $2$.
 If it does, we say that $S$ 
 {\em presents the  standard case}.
By \cite{R} we know that if $K_{S}^2>9$, 
then $S$ presents the standard case. 

Surfaces with non birational bicanonical map 
not presenting the standard case
 have been intensively investigated
in the last decade by several authors,
and their classification 
is rather understood with
 the exception of the cases when  
 $p_{g}=q\leq1$ (cf. \cite{BCP} for a survey on
 this subject).

The cases when $p_g=q=0$ have been studied
by Mendes Lopes and Pardini in the last years. 
We refer to \cite{MP8} and \cite{MP9} for a survey, here
we just mention that they obtained a complete 
classification of surfaces with $K_{S}^2\geq 6$
and nonbirational bicanonical map.
It follows from their analysis that if 
$K_{S}^2\geq 6$ and $\varphi_{2K}$
is nonbirational, then $K_{S}^2\leq 8$ and
$\varphi_{2K}$ factors through a 
double cover onto a rational surface.
Furthermore,
  $\deg \varphi_{2K}=2$ if $K_{S}^2=7,8$ and 
$\deg \varphi_{2K}=2$ or $4$ if $K_S^2=6$.

Let $S$ be a minimal surface of general type with $p_{g}=q=1$.
Then $K_{S}^2\geq2$
by \cite[Lemma 14]{bom} and $K_{S}^2\leq 9$ by \cite{Mi1} and \cite{Ya}.
The only cases for which 
there is a complete classification from the point
of view of the non birationality of the 
bicanonical map are those when $K_{S}^2=2$ 
and $K_S^2=9$.
In the first case $S$ presents the standard
case, in fact the Albanese map is a fibration
with general fiber a curve of genus $2$ (cf. \cite{C1}).
If the second case occurs, then  the bicanonical map is 
birational 
(cf. \cite[Proposition 3.1]{CM2} and  \cite[Th\'eor\`eme 2.2 and Corollaire 3]{X3}).
 
Surfaces with  $K_S^{2}=8$ have been studied by Polizzi in \cite{Po},
who considered those whose bicanonical map factors
through a double cover onto rational surface.
 He proved that in this case $\varphi_{2K}$ has 
 degree $2$ and that these surfaces  are quotient of  the product of
two curves by the action of a finite group.
 He also studied the moduli space and showed that 
 these surfaces in fact exist.

 In \cite{Po} it is also claimed that 
 if the bicanonical map of a minimal surface 
with $p_{g}=q=1$ and $K_{S}^2=8$
factors through a double cover onto 
a surface $X$, then  $X$ is rational. 
For this assertion  the author  refers to 
a theorem by Xiao (\cite[Theorem 3]{X2}), which unfortunately
has a gap, as pointed out by Rito (cf. \cite{Ri}).

\bigskip
In the present paper we prove the following.
\begin{theorem}  Ê  \label{main th}
Let $S$ be a minimal surface of general type
with $p_{g}=q=1$ and $K_{S}^2=8$.
If the bicanonical map of $S$ is nonbirational 
then $\varphi_{2K}$ is a double cover onto a rational surface.
\end{theorem}

An application of this theorem is the 
 complete classification of
surfaces with $p_{g}=q=1$,  $K_{S}^2=8$ 
and nonbirational bicanonical map (see Theorem \ref{main th2}).

\bigskip

The paper is organized as follows.
In Section \ref{sec1} we recall 
some facts which will be used in the paper. 
In Section \ref{sec2} we first fix the gap in \cite{Po}, 
and the we study the case when the bicanonical 
map has degree $4$.
In the last section, first we
 prove Theorem \ref{main th}
and then, as application, we provide the classification of 
surfaces with $p_{g}=q=1$, $K_{S}^2=8$
and nonbirational bicanonical map.

\bigskip
\subsection{Notation} 
Given a smooth surface  $X$, we denote 
by $K_{X}$ a canonical divisor of $X$ and
use the standard    notation
$p_{g}(X)=h^0(K_{X})$, $q(X)=h^1(K_{X})$ and
$\chi(X)=p_{g}(X)-q(X)+1$.
Throughout the paper $S$ is a minimal surface of general type.
For brevity sake, we  let $p_{g}=p_{g}(S)$ and $q=q(S)$.
We use the symbol $\equiv$ (resp. $\sim$) to 
denote linear  (numeric) equivalence
of divisors.
As usual, a {\em $(-n)$-curve} 
on a surface is a smooth rational curve with 
self intersection~$-n$.

A {\em fibration}  is a surjective morphism
from a smooth surface onto a smooth curve 
with connected fibers. 
A fibration is {\em isotrivial } if the smooth fibers are all isomorphic.
A fibration is said 
{\em relatively minimal} if the fibers contain no $(-1)$-curves.

Given a surface $X$ and an effective  
divisor $D\in Pic(X)$, we denote by
$\varphi_{D}:X\dashrightarrow \mathbb P^{h^0(D)-1}$ the rational map defined by 
the complete linear system $\vert D\vert$.

We will use freely the facts concerning involutions
on surfaces and on double covers between 
surfaces. We refer to \cite[section 1.1]{B}, from where we keep
some definition (for instance, of $[r,r]$-points in the proof of Proposition \ref{polizzi}).

\section{useful facts}  \label{sec1}

We gather same facts which will be used throughout. 
\smallskip
\begin{lemma}         \label{LCdF}
Let $X,Y$ be smooth surfaces and
 $X\to Y$  a double cover.
 If $p_{g}(Y)=1$ and $q(X)-q(Y)\geq 3$,
 then there exists a fibration 
$X\to C$ onto a  smooth curve of genus $g(C)\geq 2$.
\end{lemma}
\begin{proof}
Let $\sigma$ be the involution induced 
on $X$ by the double cover, 
and   $G\subseteq Aut (X)$
 the group of automorphisms 
genereted by $\sigma$,
so that $X/G=Y$. 
For any $p$ there exists
an isomorphism 
$H^0(\Omega_{X}^p)^G\cong H^0(\Omega_{Y}^p)$,
where $H^0(\Omega_{X}^p)^G\subseteq H^0(\Omega_{X}^p)$
is the subspace of $G$-invariant $p$-forms.

Let $V^-\subset H^0(\Omega_{X}^1)$ be the subspace 
of $G$-anti invariant $1$-forms, so that 
$H^0(\Omega_{X}^1)=V^-\oplus H^0(\Omega_{X}^1)^G$.
 Then $\dim V^-=q(X)-\dim H^0(\Omega_{X}^1)^G=
 q(X)-q(Y)\geq 3$ and, since $p_{g}(Y)=1$, the map 
 $\bigwedge ^2V^-\to H^0(\Omega_{X}^2)^G\cong H^0(\Omega^2_{Y})$
 is not injective.
Therefore, by the Castelnuovo-de Franchis' Theorem, there is 
a fibration
$X\to C$ where $C$ is a curve of genus $g(C)\geq 2$.
\end{proof}

We will use the following results, usually with 
no further reference.
\begin{proposition}[cf. \cite{Be1} or \cite{Be4}]       \label{beau}
Let $X$ be a smooth  surface and $f:X\to C$ a relatively minimal fibration
onto a curve of genus $b$.
Let $F$ be a general fiber of $f$ and $g$ its genus. 
Then
\begin{enumerate}
\item $K_{X}^2\geq 8(g-1)(b-1)$;
\item $12\chi (X)-K_{X}^2
\geq 4(g-1)(b-1)$;
\item $q(X)\leq g+b$.
\end{enumerate}
 
Furthermore, if equality holds in $(1)$
then the fibers of $f$ have constant modulus,
if equality holds in $(2)$ every fiber of $f$ is
smooth, and if equality holds in $(3)$ then
$X\cong F\times C$.
\end{proposition}

\begin{proposition}[cf. \cite{Se}]    \label{serr}
Let $f:X\to C$ be a relatively minimal isotrivial fibration, with
general fiber a curve  $F$ of genus $g\geq1$.
Then there exist a curve $A$ and a finite group
$G$ acting faithfully and holomorphically on
$A$ and $F$, such that $C\cong A/G$, and there is
a birational map $\pi:X\to (A\times F)/G$
(where $G$ acts diagonally on $A\times F$)
for which the following diagram commutes
$$
  \xymatrix{
 X  \ar[rr]^\pi  \ar[dd]_{f}    
&& (A\times F)/G \ar[dd]
\\
\\
B\ar[rr]^\cong && A/G 
}
$$
Furthermore,
\begin{enumerate}
\item 
$q(X)=g(A/G)+g(F/G)$;
\item
the natural map 
  $(A\times F)/G\to F/G$ give rise to
an  isotrivial fibrations
   and  $f_{1}:X\to F/G$ whose general fiber is isomorphic 
   to $A$;
\item
 if the only singular fibers of $f$ are multiples 
of smooth curves, then $\pi$ is an isomorphism and
$G$ operates on $A\times F$ without fixed points
(i.e. $(A\times F)\to X$ is \'etale).
\end{enumerate}
\end{proposition}

\begin{lemma}  \label{lmX}
Let $S$ be minimal surface of general type with $p_{g}=q=1$ 
and $K_{S}^2=8$.  
Let $S\to C$ be a fibration. Let $b$ the genus of $C$ and $g$ the genus
of a general fiber.

Then  $g\geq3$ $($\cite[Th\'eor\`eme 2.2 and Corollaire 3]{X3}$)$
and $b\leq 1$ $($cf. for instance \cite{Be4}$)$.
In particular, if $b=1$ then $C$ is the Albanese variety 
of $S$ and $S\to C$ is the Albanese morphism.
\end{lemma}
\section{
The bicanonical map of surfaces with $K_{S}^2=8$
}   \label{sec2}
Throughout this section let $S$
be a minimal surface of general type with
$p_{g}=q=1$, $K_{S}^2=8$ and nonbirational bicanonical map.
Let 
$\varphi _{2}=\varphi_{2K_{S}}:S\dashrightarrow S_{2}\subset \pp^8$ 
be the  bicanonical map  of $S$.
Then $\varphi _{2}$ is a morphism (cf. \cite{R}) 
and, since  by \cite{Mi2} $K_{S}$ is ample,
it is finite. Let $d\ge 2$ be the degree of $\varphi _{2}$ and $\text{deg}(S_{2})$
the degree of $S_{2}$. Then $(2K_{S})^2=d\cdot \text{deg}(S_{2})$
and so $d=2$ or $4$, because an irreducible non degenerate surface in $\mathbb P^n$ has 
degree at least $n-1$.

\subsection{The case when $\varphi_{2}$ factors through an involution.}
The following proposition fix the gap of \cite{Po} (cf. the introduction). 

\begin{proposition}  \label{polizzi}
Let $S$ be a minimal surface of general type
with $p_{g}=q=1$ and $K_{S}^2=8$. 
Suppose that the bicanonical map of $S$ factors through 
an involution $\sigma$. 
Then the quotient $S/\sigma$ is birational to a rational surface.
\end{proposition}
\begin{proof}
Let $\hat \Sigma\to S/\sigma$  be the minimal desingularization,
and suppose 
that $\hat\Sigma$ is not rational.
Then, by \cite[Theorem 2.7]{B}, $\hat\Sigma$ has non negative
Kodaira dimension and thus,
since $S$ does not present the standard case,
it is birational to a $K3$ surface (cf. \cite{Ri}). 

Let $\pi:\hat S\to S$ the blow up of the 
isolated fixed points of $\sigma$, and $\hat \sigma$  
the involution induced on $\hat S$.
Then   the quotient $\hat S/\hat \sigma$ is isomorphic to $\hat \Sigma$ and 
 we have the following commutative diagram 
$$
  \xymatrix{
\hat S  \ar[r]^\pi  \ar[d]_{\rho}  
& S \ar[d]
\\
%\\
\hat \Sigma\ar[r]& S/\sigma 
}
$$
where $\rho$ is the canonical projection.
Let $E_{1},\dots,E_{\nu}$ be the exceptional 
$(-1)$-curves of $\pi$, and $C_{i}=\rho(E_{i})$.
Let $\hat B
%=\rho(\hat R+\sum E_{i})
$ 
be the branch curve of $\rho$. 
Then  $C_{i}$ is  a $(-2)$-curve 
 belonging to $\hat B$ for every $i$.

Let $\phi:\hat \Sigma\to \Sigma$ be the birational
morphism to the minimal model, and $B=\phi_{\ast } \hat B$. 
Then $B\equiv 2\Delta$ for some $\Delta \in Pic(\Sigma)$ 
and, by \cite[proof of Theorem 1]{Ri}, $B$ has 
at most one essential singularity, which is either a $[4]$-point or a  $[3,3]$-point.

Let $\calE_{1},\dots,\calE_{d}$ be the reduced and irreducible 
exceptional curves of $\phi$. Since $\Sigma$ is 
a $K3$ surface, $K_{\hat \Sigma}=\sum e_{j}\calE_{j}$,
where the $e_{j}$'s are positive integers.
Since  the $C_{i}$'s are $(-2)$-curves, we have 
$C_{i}\cdot \sum e_{j}\calE_{j}=0$ for every $i$.
Hence either $C_{i}\cap \sum  \calE_{j}=\emptyset$
or $C_{i}=\calE_{j}$ for some $j$, i.e.
either $\phi(C_{i})$ is a $(-2)$-curve or 
$\phi(C_{i})$ is a point.
Thus,  
by \cite[Lemma 1.4]{B}, there are at least $\nu-1$ pairwise 
disjoined $(-2)$-curves on $\Sigma$ and then
$\nu\leq 17$, because $\Sigma$ is a $K3$ surface.

By assumption, $\varphi_{2}$ factors through $\sigma$. 
Then, by \cite[Proposition 1.8]{B},
we get $ \nu=K_{S}^2-2\chi(S)+6\chi(\Sigma)=18$,
which is a contradiction.
Therefore $S/\sigma$ is birational to a rational surface.
\end{proof}

\subsection{The case when $d=4$}
In this section we prove the following
\begin{theorem}   \label{thm d4}
Let $S$ be a minimal surface of general type with
$p_{g}=q=1$ and $K_{S}^2=8$. 
Suppose that the degree of the bicanonical map is $4$.
Then $\varphi_{2}$ factors through a 
generically finite map of degree two onto 
a rational surface.
\end{theorem}

Assume 
$d=4$.  
 Then $S_{2}$ is a surface of degree $8$ in $\pp^8$, and thus 
 it is one of the following (cf. \cite {Na}).
 \begin{itemize}
\item
A cone over an elliptic curve of degree $8$ in $\mathbb P^7$.
\item
The Veronese embedding in $\mathbb P^8$ of 
a quadric $Q$ in $\mathbb P^3$.
\item
A del Pezzo surface of degree $8$, i.e.
the image of $\mathbb P^2$ by the rational map
associated to the linear system 
$\vert \calO_{\mathbb P^2}( 3)\otimes \mathcal I_{x}\vert$,
where $\mathcal I_{x}$ is the ideal sheaf of
a point in $\mathbb P^2$.
\end{itemize}

We will consider each case separately.
In particular, we are going to show that the first case does 
not occur  and that in both  the others $\varphi_{2}$ factors 
through a double cover onto  a rational surface.  

Theorem \ref{thm d4} follows from Lemma \ref{lm0}, 
Propositions \ref{prop1} and \ref{lm01} below.
\begin{lemma}      \label{lm0}
$S_{2}$ is not a cone over an elliptic curve.
\end{lemma}
\begin{proof}
Suppose that $S_{2}$ is a cone over an elliptic curve $C$.
Let $\mu:S_{2}\dashrightarrow C$ be the projection
 from the vertex.  
The composite map $f=\mu\circ \varphi_{2}:S\dashrightarrow C$
is surjective and thus a morphism, 
since $C$ has positive genus.

The Stein factorization of $f$ yields a fibration
$\tilde f:S\to \tilde C$ over a smooth curve of genus 
$g(\tilde C)\geq 1$, whose general fiber is a curve $F$
of genus $g(F)$. 
Since $K_{S}^2=8$, by Lemma \ref{lmX}
 we have $g(F)\geq 3$
and thus $F\cdot K_{S}\geq 4$.

Let $\ell=\varphi_{2}(F)$.
Then $\ell$ is a general line on $S_{2}$
and, since $K_{S}$ is ample,  we have 
$\varphi_{2}^{\, \ast} \ell\sim aF$, where 
$a\geq1$ is an integer.
But then we get 
$4=(2K_{S})\cdot (aF)
\geq 8a$,
which is absurd.
\end{proof}

\begin{proposition}     \label{prop1}
Suppose that $S_{2}$ is
the Veronese embedding in $\mathbb P^8$ of 
a quadric $Q$ in~$\mathbb P^3$.
Then the bicanonical map of $S$ factors through a double
cover onto a rational surface.
\end{proposition}

\begin{proof}
Let $H$ be the pull back to $S$ of a general
hyperplane section of $Q$, so that
 $2H\equiv 2K_{S}$.
Let $\eta= H-K_{S}$. 
Then $2\eta\equiv 0$ and so,
since $h^0(K_{S}+\eta)=h^0(H)=4\neq 
h^0(K_{S})$,
$\eta$~is a non zero $2$-torsion element.

\ %bigskip

Let $\psi:X\to S$ be the \'etale double cover
defined by $2\eta\equiv 0$. 
We have $\chi(X)=2\chi(S)=2$, $K_{X}^2=2K^2_{S}=16$ and
 $p_{g}(X)=p_{g}(S)+h^0(K_{S}+\eta)=5$.
Whence  $q(X)=4$ and thus, 
by Proposition \ref{LCdF}, there exists
a fibration $f:X\to C$ over a curve of genus 
$g(C)\geq 2$.
Let $F$ be a general fiber of $f$ and $g(F)$ its genus.
 Since $X$ is of general type, $g(F)\geq 2$.
Observe that, since $\psi:X\to S$ is \' etale and
$S$ is minimal, $X$ is minimal. 
In particular, the fibration
$f:X\to C$ is relatively minimal.\ %bigskip

\

{\bf Claim. }{\em 
The fibers of $f:X\to C$ are smooth and isomorphic curves of genus~$3$,
and  we have $g(C)=2$.}

%\

Suppose  $g(F)=2$.
Since $q(X)=5$, by Proposition \ref{beau} we get $g(C)\geq 3$.
Let $\sigma$ be
the involution on $X$ induced by the double cover $\psi:X\to S$.
Since $F$ does not dominate $C$, $\sigma (F)$
is a fiber of $f$.
 Then there are a fibration $f^\prime:S\to C^\prime$
over a smooth curve whose general fiber is  $\psi(F)$
 and a finite morphism $h:C\to C^\prime$
such that $f^\prime \circ\psi=h\circ f$.
Since the genus of $\psi(F)$ is at most $g(F)=2$,
we  contradicts 
Lemma \ref{lmX}. 
Therefore, $g(F)\geq 3$. 

By Proposition \ref{beau}, 
we have 
$
16=K^2_{X}\geq 8(g(F)-1)(g(C)-1)
$
which yields
  $g(F)=3$ and $g(C)=2$.
Since $12\chi(X)-K_{X}^2=8=4(g(F)-1)(g(C)-1)$, 
the fibers are all smooth and isomorphic.
This proves the claim.

\

By Proposition \ref{serr}
 there are a curve $A$ and a finite group $G$
acting faithfully and holomorphically on $F$ and $A$,
such that
$X\cong (A\times F)/G$.
Furthermore, there exists a fibration
$f_{1}:X\to F/G$ whose general fiber, say $\hat A$, is ismorphic to $A$.
In paticular,  we have $F\cdot \hat A=\vert G\vert$.

Let $C_{1}=F/G$.
Since $4=q(X)=g(A/G)+g(F/G)=2+g(C_{1})$, 
we get $g(C_{1})=2$  
and the same argument we used for $f:X\to C$
yields $g(A)=3$.

We have $p_{g}(A\times F)=g(F)g(A)=9$ and 
$q(A\times F)=g(F)+g(A)=6$. 
Thus $\chi (A\times F)=4$ and so,
since $A\times F\to X$ is \' etale, 
$
\vert G\vert = 2.
$
Hence $\hat A\cdot F=2$ and 
then $f_{1}$ (resp. $f$) induces a double cover
$F\to C_{1}$ (resp. $A\to C$).
It follows that
$F$ (resp. $\hat A$) is a hyperelliptic curve (cf. \cite[Lemma 5.10]{Accola}).

\bigskip

Let $\tau$ 
 be the involution
on $X$ induced by the hyperelliptic involution
of $F$, and  let
 $Y=X/\tau$ be
 the quotient. 
Let $\sigma$ be the involution induced 
 by the double cover $X\to S$,
and let $N\subseteq Aut (X)$ %(resp. $B_1$)
be the subgroup generated by 
 $\tau$ %(resp. $\tau_{1}$) 
 and $\sigma$.
We have a commutative diagram
$$
  \xymatrix{
& X %\cong (A\times F)/G
 \ar[ld]_{\psi}  \ar[rd]^\phi 
%\\
\\
S=X/\sigma   \ar[rd]_{\tilde\psi}& & Y=X/\tau  \ar[ld]^{\tilde\phi}
%\\
\\
& X/N&
}
$$
where $\phi$ is the canonical projection, and
$\tilde \psi,\tilde \phi$ are double covers.
Note that, since the fibers of $f:X\to C$ are all isomorphic,
$\tau$ has not isolated fixed point and hence $Y$ is smooth.

By construction, there is a fibration $f^\prime :Y\to C$ whose general
fiber is $\phi(F)$, which is a smooth rational curve.
In particular, $Y$ is not birational to  $S$.
Note that since $\phi(A)$ dominates $C$, 
$\phi(A)$ is not rational.

\

Let $\tilde \sigma$ be the involution induced on $Y$ by the 
double cover $Y\to X/N$.
Since $\phi(F)$ does not dominate $C$, 
we have that $\tilde \sigma(\phi(F))$
is a fiber of $f^\prime:Y\to C$.

Then there is a fibration $\tilde f:X/N\to \tilde C$ and
a finite morphism $h:C\to \tilde C$
such that $h\circ f^\prime= \tilde f \circ \tilde \phi$ and
$\tilde \phi(\phi(F))$ is a fiber of $\tilde f$.
Let $\Gamma=\tilde \phi(\phi(F))$.

The composite map $\tilde f\circ \tilde \psi=S\to \tilde C$ is
a surjective morphism with general fiber the pull back of
$\Gamma$.
Since $\tilde \psi $ is a double cover and $\Gamma$ is rational, 
  $\tilde \psi^\ast \Gamma$ is connected,
 i.e. $\tilde f\circ \tilde \psi$ is a fibration.
 By Lemma \ref{lmX} we have
$g(\tilde C)\leq 1$.

By the commutativity of the diagram we have 
$\tilde \psi^\ast\Gamma=\psi (F)$.
Since $\psi$ is \'etale and on $S$ there are no pencils of genus
two curves, $\psi(F)\cong F$ is a hyperelliptic curve of genus $3$.

\

By applying the same argument to the fibration $f_{1}:X\to C_{1}$
we get a fibration $S\to \tilde C_{1}$ such that $g(\tilde C_{1})\leq 1$
 whose general fiber is a smooth hyperelliptic curve of genus
$3$ isomorphic to $\hat A$.

It is easy to see that these two fibrations are not the same.
 Therefore, since on $S$ there exists exactly one fibration over 
an elliptic  curve, namely the Albanese fibration,
we have $C\cong \pp^1$ or $C_{1}\cong \pp^1$.
By \cite{B1}  the bicanonical map of $S$ factors through 
a double cover onto a rational surface.
\end{proof}

Finally  we consider the last case.
%
%\ %bigskip
%
To begin with, let us introduce some notation.

\

From now on we assume that 
$S_{2}$ is 
the image of $\mathbb P^2$ in $\pp^8$ by the birational map
associated to the linear system 
$\vert \calO_{\mathbb P^2}( 3)\otimes \mathcal I_{x}\vert$,
where $\mathcal I_{x}$ is the ideal sheaf of
a point $x\in \mathbb P^2$.
Let $\hat \pp\to \pp^2$ be the blow up of $x$,
$\ell\subset \hat \pp$ the pull back of a general line in  $\pp^2$,
and $\hat \ell$ the strict transform of a general line through $x$.
Then $S_{2}$ is the image of $\hat \pp$
by the birational morphism $\mu:\hat \pp\to \pp^8$
 associated to  $\vert 2\ell+\hat \ell\vert$.
 Note that $\vert 2\ell+\hat \ell\vert=\vert- K_{\hat \pp}\vert$.

Let $L$ and $\hat L$ be the pull back to $S$ of
$\mu_{\ast }\ell$ and $\mu_{\ast}\hat \ell$, respectively, so that 
$2K_{S}\in\vert2L+\hat L\vert$.
We have $\hat L^2=0$ and $ \hat L\cdot K_{S}=4$.
Since $\hat L$ is a fiber of the composite morphism 
$\varphi_{\hat \ell}\circ \varphi_{2}:S\to \pp^1$,
it follows from Lemma \ref{lmX} that $\hat L$ is a 
smooth connected curve of genus $3$. 
In particular, $\varphi_{\hat L}=\varphi_{\hat \ell}\circ \varphi_{2}:S\to \pp^1$
is a fibration.

\

Fix a smooth curve $B\in\vert \hat L\vert$, and
 let $\psi :X\to S$ be the double cover defined by 
$B\equiv 2(K_{S}-L)$. 
Then $X$ is a smooth surface with $\chi(X)=3$
and $K_{S}^2=24$.
 Note that, since $S$ is minimal of general type and the branch
 curve is a smooth curve of genus $3$, 
 $X$ is a minimal surface of general type.

In order to compute the invariants of $X$ we observe
that $\varphi_{2}(\ell)$
is  a twisted cubic in $\pp^3$.
Then $h^0(2K_{S}-L)=h^0(2K_{S})-4=5$, and so
$$p_{g}(X)=h^0(K_{S}+(K_{S}-L))+p_{g}(S)=6.$$

Whence $q(X)=4$ and thus, by 
Proposition \ref{LCdF}, there is 
a fibration $f:X\to \mathbb B$ onto a smooth curve $\mathbb B$ of
genus $g(\mathbb B)\geq 2$.
Let $F$ be a general fiber of $f$ and $g$ its genus.
By Proposition \ref{beau} we have the following possibilities
\begin{equation}   \label{eq2}
\begin{aligned}
%\item $
g(F)= 4 &\, \text { and } g(\mathbb B)=2; \\ %$;
%\item $
g(F)=3 &\, \text { and } g(\mathbb B)=2; \\ %$;
%\item $
g(F)=2 &\, \text { and } 2\leq g(\mathbb B)\leq 4. %$.
\end{aligned}
\end{equation}

\

Consider the composite morphism
 $\varphi_{\hat L}\circ\psi:X\to \pp^1$.
 One of the following cases occurs.
 \begin{itemize}
 \item[$(\dagger)$]
 $  h_{1}=\varphi_{\hat L}\circ \psi:X\to \pp^1$ is a fibration whose 
 general fiber is $\psi^\ast \hat L$;
  \item[$(\ddagger)$]
there exist a fibration $  h_{2}:X\to C$
and a double cover $\tilde \psi:C\to \pp^1$ such that
$h_{2}\circ \tilde \psi=\varphi_{\hat L}\circ \psi$ and
 a general fiber of $h_{2}$ is isomorphic to $\hat L$.
 \end{itemize}

Let $M$ be a general fiber of $ h_{i}$, so that
$g(M)=5$ in the case $(\dagger)$ and
$g(M)=3$ in the case $(\ddagger)$.

\

We will proceed by analyzing 
these two cases separately.
We will show that the first one does not occur and 
that in the second case $\hat L$
is hyperelliptic.
Before going further let  
us observe that in any case, since 
$$
\left(
\frac {F}{g(F)-1}+\frac{M}{g(M)-1}-\frac{K_{X}}{6}
\right)\cdot K_{X}=0,
$$
by the algebraic index theorem
 we have
\begin{equation}  \label{eq1}
3(F\cdot M)\leq
(g(F)-1)(g(M)-1),
\end{equation}
and  the equality holds if and only if 
$$
6(g(M)-1)F+6(g(F)-1)M \sim (g(F)-1)(g(M)-1)K_{X}.
$$

\

\begin{lemma}   \label{not}
The case $(\dagger)$ does not occur.
\end{lemma}
\begin{proof}
To rise a contradiction, suppose that $h_{1}$ is
a fibration.
The curve $\psi^\ast B$
is a double fiber of $ h_{1}$. 
Then $D\cdot M\equiv 0\pmod 2$ for any
 divisor  $D\in Pic (X)$ and
in particular, by \eqref{eq2} and \eqref{eq1}, 
we have
$F\cdot M\in\{0,2,4\}$.

\

Suppose $F\cdot M=4$. 
Since $g(M)=5$ and $g(F)\leq 4$ (cf. \eqref{eq2}),  \eqref{eq1} yields $g(F)=4$.
Then
\begin{equation}   \label{eqFMK}
4F+3M\sim 2K_{X}.
\end{equation} 
By \eqref{eq2}, we have $g(\mathbb B)=2$ and so,
by Proposition \ref{beau}, $f:X\to \mathbb B$ is isotrivial and smooth.  
Then, there exist a curve $A$ and a group $G$
such that $(F\times A)/G$.
Since $q(X)=g(F/G)+g(A/G)=g(F/N)+g(\mathbb B)$,
we have $g(F/G)=2$. 

Let  $\mathbb B_{1}=F/G$ and let $f_{1}:X\to \mathbb B_{1}$
 be the natural projection.
Then $f_{1}$ is an isotrivial fibration whose  general fiber, say $\hat A$,
 is isomorphic to $A$. 
Note that, since $g(\mathbb B_{1})=2$, 
Proposition \ref{beau} yields $g(A)\leq 4$.
Hence by \eqref{eqFMK} we get
 $$
4\vert G\vert +3M\cdot \hat A=(4F+3M)\cdot \hat A= 2K_{X}\cdot \hat A=4(g(A)-1)\leq
12.
 $$

Observe that  $\vert G\vert>1$, for otherwise
 $X\cong F\times \mathbb B$ and thus 
 $q(X)=g(F)+g(\mathbb B)=6$,
which is a contradiction.

It follows that $\hat A\cdot M\leq 1$ 
and then, since  $\hat A\cdot M\equiv 0\pmod 2$,
%which yields
$\hat A\cdot M=0$.
  Thus $M$ is a fiber of $f_{1}$
and so it is isomorphic to $A$. 
But this is absurd, because $g(M)=5$
and $g(A)\leq 4$.
Then $F\cdot M\neq 4$.

\bigskip

Since $g(M)=5$ and $g(F)\leq 4$, 
$F\cdot M=0$ cannot occur.
Therefore we are left with   $F\cdot M=2$. 
Since $\psi^{\ast}B$ is a double fiber of 
$h_{1}$, 
  $F\cdot \psi^{-1}(B)=1$
and so the restriction of $f:X\to \mathbb B$ 
to $B$ is an isomorphism.
Thus $g(\mathbb B)=g(B)=g(\hat L)=3$ and, 
by \eqref{eq2}, we get $g(F)=2$. 
But this contradicts \eqref{eq1}.
Therefore $F\cdot M\neq 2$,
 and so
case $(\dagger)$ does not occur.
\end{proof}
\

\begin{lemma}  \label{lm01}
Suppose that $S_{2}$ is a del Pezzo surface of degree $8$.
Then $\varphi_{2}$ factors through a double cover
 onto a rational surface.
\end{lemma}
%.
\begin{proof}
By Lemma \ref{not} we have  a fibration  $h_{2}:X\to C$ 
whose general fiber $M$ is a smooth curve of genus $3$
isomorphic to $\hat L$, and a double cover $\tilde \psi:C\to \pp^1$
such that $ \tilde \psi\circ h_{2}=\varphi_{\hat L}\circ \psi$.

By \eqref {eq1} and   \eqref {eq2} we have 
$3(F\cdot M)\leq 2(g(F)-1)\leq 6$.

\

Suppose $F\cdot M=0$. Then $M$ and $F$ are fibers of the same fibration.
Hence $g(F)=g(M)=3$ and, by  \eqref {eq2}, $C=\mathbb B$ is a curve of genus $2$.
By Rieman-Hurwitz for the double cover $\tilde \psi:C\to \pp^1$,
 the fibration $\varphi_{\hat L}:S\to \pp^1$ has at least 
$(2b+2)-1=5$ double fibers,
say $2\Delta_{1},\dots,2\Delta_{5}$.
The curves  $\hat \ell_{i}=\varphi_{2}(2\Delta_{i})\in\vert \hat\ell\vert$ 
are reduced, irreducible and distinct.
Since $2\Delta_{i}=\varphi_{2}^{\, \ast} \hat \ell_{i}$,
  $\Delta_{i}$ is a component 
of the ramification divisor of $\varphi_{2}$
 for every $i\in\{1,\dots,5\}$.

\

Let $R$ be the ramification divisor of $\varphi_{2}:S\to S_{2}$.
Denote by $R_{i}$ the reduced and irreducible components
of $R$ such that $B_{i}=\varphi_{2}(R_{i})$ is a curve, 
and let $E_{j}$ be the components of $R$ which are contracted
by $\varphi_{2}$. 
Then (cf.  \cite[Lemme 3.2]{Be3})
$$
K_{S}=\varphi_{2}^{\, \ast} K_{S_{2}} +\sum (r_{i}-1)R_{i}+\sum e_{j}E_{j},
$$
where $e_{j}\geq 0$ is an integer and
 $r_{i}\geq 2$ is the ramification index of $\varphi_{2}$
at a generic point of $R_{i}$,
i.e. the coefficient of $R_{i}$ in $\varphi_{2}^\ast B_{i}$.
Since $K_{S}$ is ample we have $\sum E_{j}=0$.
Whence, since $2K_{S}\equiv\varphi^\ast_{2}(-K_{S_{2}})$, we get 
$$
 3K_{S}\equiv \sum (r_{i}-1)R_{i}.
$$

\

Let $e\subset S_{2}$ be the the image of the exceptional curve of $\mu_{x}$,
and let $E=\varphi^\ast e$. Then $E^2=-4$ and $E\cdot K_{S}=2$.
By  \cite{Mi2} there are no $(-n)$-curves on $S$.
Hence, since $K_{S}$ is ample,
 $E=A_{1}+A_{2}$, where $A_{1}$, $A_{2}$ are 
irreducible curves with  $A_{i}\cdot K_{S}=1$ 
and, 
 by adjunction and the algebraic index
theorem, 
%we have 
$A_{i}^2=-1$, $i=1,2$.
Then $(A_{1}+A_{2})^2=-4$ yields $A_{1}\cdot A_{2}=-1$,
i.e. $A_{1}=A_{2}$. 
In particular, 
$A_{1}$ is a component of   $R$ with ramification index $2$.

It follows that the divisor
$$
3K_{S}-A_1-\sum_{i=1}^5 \Delta_{i} 
$$
is effective and  $A_{1}$ does not belong to it.
But then,  since
$4\Delta_{i}\cdot A_{1}=\hat L\cdot E=4 (\hat \ell\cdot e)=4$
for every $i$, we get
 $$
(3K_{S}-A_{1}-\sum \Delta_{i})\cdot A_{1}=3+1-5<0,
$$
 a contradiction.
Therefore, $F\cdot M\neq0$.

\

Suppose $M\cdot F=1$.
Thus $M$ is isomorphic to $\mathbb B$, and hence $g(\mathbb B)=3$.
But then, by \eqref{eq2}, we get  $g(F)=2$ which contradicts \eqref{eq1}.

\

Therefore we are left with   $M\cdot F=2$. 
 The restriction of  $f:X\to \mathbb B$ to $M$ is
 then a double cover $M\to \mathbb B$.
Hence, since $M$ is a smooth curve of genus $3$
and $\mathbb B$ has genus $2$,
$M$ is hyperelliptic [{\em loc.cit\, }].

Since  $\hat L$ is isomorphic to $M$,
$\varphi_{\hat L}:S\to \pp^1$ is fibration with 
general fiber a  hyperelliptic curve of genus $3$.
By \cite{B1} the bicanonical map of $S$ factors through a
double cover onto a rational surface.
This completes the proof.
\end{proof}

\

\section{Conclusion}  \label{sec3}

\

First of all we prove Theorem \ref{main th}.

\begin{proof}[Proof of Theorem \ref{main th}]
Let $d$ be the degree of $\varphi_{2}$.
 Then $d=2$ or $4$.

Suppose $d=2$. Let $\sigma$ be the involution induced by 
$\varphi_{2}:S\to S_{2}$. Then $\varphi_{2}$ factors through
$\sigma$ and $S_{2}$ is birational to $S/\sigma$.
Form Lemma \ref{polizzi} it follows that $S_{2}$ is rational.   

Suppose that $d=4$. Then, by Theorem \ref{thm d4},
$\varphi_{2}$ factors through a double cover onto a rational surface.
But this contradicts \cite{Po}.

Therefore, $d=2$, $S_{2}$ is  a rational surface
 and, since $K_{S}$ is ample, $\varphi_{2}$ is
finite, i.e. a double cover.
\end{proof}

\

The following theorem characterizes completely, 
from different points of view, minimal surfaces
of general type with $p_{g}=q=1$, $K_{S}^2=8$
and non bicanonical map. 

\begin{theorem} \label{main th2}
Let $S$ be a
 minimal surface
of general type with $p_{g}=q=1$ and $K_{S}^2=8$.
The following are equivalent.
\begin{enumerate}
\item
The bicanonical map is non birational.
\item 
The bicanonical map is a double cover onto a rational surface.
\item
 $S$ is birational to a du Val double 
plane of type $\mathcal D_{6}$ $($cf. \cite{B}$)$, i.e.
the smooth minimal model of a 
double cover of $\pp^2$ branched 
 over the union $B$ of a curve of degree $16$ 
with $6$ distinct lines through a point $p$,
 such that the essential singularities
of $B$  are the following:
\begin{itemize}
\item [ - ] $p$ is a singular point of multiplicity $14$, 
\item [ - ] there is a  $[5, 5]$-point on each line.
\end{itemize}
\item  
$S$ is isomorphic to the quotient of a product 
$F\times C$ of two smooth curves of genus $g(F)$ and $g(C)=3$
by a group $G$, where either
\begin{itemize}
\item [ - ] $g(F)=3$ and $G= \mathbb Z_{2}\times \mathbb Z_{2}$, or
\item [ - ]  $g(F)=4$ and $G=S_{3}$, or 
\item [ - ]  $g(F)=5$ and $G=D_{4}$.
\end{itemize}
\item 
There exists a fibration $S\to \pp^1$ whose general 
fiber is a smooth hyperelliptic curve of genus $3$.
\end{enumerate}
\end{theorem}
\begin{proof}
Implication $(1)\Rightarrow (2)$ follows from Theorem \ref{main th}, and 
$(2)\Rightarrow (1)$ is trivial.
The implications $(2)\Rightarrow (3)$ and $(3)\Rightarrow (1)$
are proved in \cite{B}.  
The equivalence $(4)\Leftrightarrow (2)$ follows from \cite{Po}.
Finally, $(1)\Leftrightarrow (5)$ follows from \cite{B1}.
\end{proof}

We remark that this classification is effective, that is these surfaces do 
exist. In fact,   in \cite{Po} Polizzi proves that they form three components of the moduli
space, one of dimension $5$ and two of dimension $4$.
He also constructs examples as in $(3)$, i.e.
quotient  of a product of curves.


\begin{thebibliography}{Hor}%I-III}
%\small {
\bibitem{Accola} R. Accola {\em Topics in the theory of the Riemann surfaces}, 
                         Lecture Notes in Mathematics 1595 (1994),
\bibitem{Be3}  A. Beauville, {\em L'application canonique pour les surfaces de type g\'en\' eral},
                                              Inventiones Math. {\bf 55} (1979), p.121-140.
\bibitem {Be1} A. Beauville, Appendice to {\em L'in\'egalit\'e $p_{g}\geq 2q-4$ pour les superfices
                 de type g\'en\'eral.} Bull. Soc. Math. France {\bf 110} (1982), p. 343-346. 
\bibitem {Be4}  A. Beauville, {\em Complex algebraic surfaces,} Cambridge University Press 1996.
\bibitem {bom} E. Bombieri, {\em Canonical models of surfaces of general type},
              Inst. Hautes Études Sci. Publ. Math., {\bf 42} (1973), p. 171-219.
\bibitem {B}  G. Borrelli, {\em The classification of surfaces of general type 
                with non birational bicanonical map.} J. Algebraic Geom. {\bf 16} (2007), 625-669. 
\bibitem {B1}  G. Borrelli, {\em  Bicanonical map of surfaces with $\chi=1$ fibered by 
                    hyperelliptic curves of genus ~$3$.} (submitted).
\bibitem{C1}  F.Catanese, {\em On a class of surfaces of genera type}.
              in Algebraic Surfaces, CIME, Liguori (1991), 269-284.
\bibitem{CM2} C.Ciliberto, M. Mendes Lopes, {\em On surfaces with $p\sb g=q=2$ and
              non-birational bicanonical map}, Adv. Geom. $\mathbf 2$ (2002), no.3, 281--300.
\bibitem {BCP} I.Bauer, F. Catanese, R. Pignatelli, {\em Complex surfaces of general type: some recent
                              progress.} in 'Global aspects of
                                    complex geometry' F.Catanese et al. editors, Springer Verlang.
\bibitem{C2}  F.Catanese,  {\em   Canonical rings and ÓspecialÓ surfaces of general type}.
                                 In:  Algebraic geometry, Bowdoin, 1985 (Brunswick, Maine, 1985), 175-194.
                                 Proc. Sympos. Pure Math., {\bf 46}, Part 1. Amer. Math. Soc., Providence, RI (1987).
\bibitem {D1}  O.Debarre, {\em In\'egalit\'es num\'erique pour les surfaces de type g\'en\'eral},
              Bull. Soc. Math. France {\bf 110} (1982), 319-346.          
\bibitem{MP8} M.Mendes Lopes, R.Pardini, {\em A survey on the bicanonical map of 
              surfaces with $p_{g} = 0$ and $K\geq2$.} In: Beltrametti, M.C. (ed) et al. Algebraic 
              geometry. A volume in memory of Paolo Francia, 277-287. De Gruyter, Berlin (2002).
\bibitem{MP9} M.Mendes Lopes, R.Pardini, {\em The degree of the bicanonical map of a 
                              surface with $p_{g}=0$.} Proc. Amer. Math. Soc. {\bf 135} (2007), 1279-1282.
\bibitem {Mi1}  Y. Miyaoka, {\em On the Chern numbers of surfaces of general type.}
                             Invent. Math. {\bf 42},  (1977), 225-237.
\bibitem {Mi2}  Y. Miyaoka, {\em The maximum number of quotient singularities on 
                           surfaces with given numerical invariants.} 
                            Math. Ann. {\bf 268}, (1984), 159-171.                              
\bibitem{Na} M. Nagata, {\em On rational surafaces I}. Mem. Coll. Sci. Univ. Kyoto {\bf 32},
                       351-370 (1960).
\bibitem{Po}  F.Polizzi, {\em Surfaces of general type with $p_g=q=1, K^2=8$ 
                 and bicanonical map of degree 2.} Trans. Amer. Math. Soc. {\bf 358}, (2006), 759-798. 
\bibitem {R}  I.Reider, {\em Vector bundles of rank 2 and linear systems on 
                         algebraic surfaces. }Ann. of Math., {\bf 127} (1988).
\bibitem {Ri}  C. Rito,   {\em On surface with $p_{g}=q=1$ and non-ruled bicanonical involution.}  
                         Annali della Scuola Normale Superiore di Pisa. Classe di scienze,  {\bf 6}, (2007),
                         81-102. 
\bibitem {Se}   F. Serrano, {\em Isotrivial fibred surfaces}, Annali di Matematica pura e applicata
                                      vol. CLXXI (1996), 63-81. 
\bibitem {X2} G.Xiao, {\em Degree of the bicanonical map of a surface of general type}, Amer. J. of Math.,
              {\bf 112} (5) (1990), 713-737.
\bibitem {X3} G.Xiao, {\em   Surfaces Fibr\'es en Courbes de Genre Deux}, Lecture Notes in Math.
                     {\bf 1137}, (1985).
\bibitem {Ya} S. T. Yau, {\em On the Ricci curvature of a compact Kahler manifold and the
                                           complex Monge-Amp\`ere equation.} 
                                           I. Comm. Pure Appl. Math. {\bf 31}, no. 3,  (1978),339-411. 
%}
\end{thebibliography}
\end{document}